\documentclass[11pt]{article}
\usepackage[margin=1in]{geometry} 
\usepackage{amssymb,amsfonts,amsmath,bbm,mathrsfs,stmaryrd}
\usepackage{xcolor}
\usepackage{url}

\usepackage{enumerate}

\usepackage[colorlinks,
             linkcolor=black!75!red,
             citecolor=blue,
             pdftitle={},
             pdfproducer={pdfLaTeX},
             pdfpagemode=None,
             bookmarksopen=true
             bookmarksnumbered=true]{hyperref}
            
\usepackage{tikz}
\usetikzlibrary{arrows,calc,decorations.pathreplacing,decorations.markings,intersections,shapes.geometric,through,fit,shapes.symbols,positioning,decorations.pathmorphing}

\usepackage{braket}

\usepackage[amsmath,thmmarks,hyperref]{ntheorem}
\usepackage{cleveref}

\creflabelformat{enumi}{#2(#1)#3}

\crefname{section}{Section}{Sections}
\crefformat{section}{#2Section~#1#3} 
\Crefformat{section}{#2Section~#1#3} 

\crefname{subsection}{\S}{\S\S}
\AtBeginDocument{%
  \crefformat{subsection}{#2\S#1#3}%
  \Crefformat{subsection}{#2\S#1#3}%
}

\theoremstyle{plain}

\newtheorem{lemma}{Lemma}[section]
\newtheorem{proposition}[lemma]{Proposition}

\newtheorem{theorem}[lemma]{Theorem}

\newtheorem{question}[lemma]{Question}

\theoremstyle{nonumberplain}

\theoremstyle{plain}
\theorembodyfont{\upshape}
\theoremsymbol{\ensuremath{\blacklozenge}}

\newtheorem{definition}[lemma]{Definition}

\newtheorem{remark}[lemma]{Remark}
\newtheorem{convention}[lemma]{Convention}
\newtheorem{notation}[lemma]{Notation}

\crefname{definition}{definition}{definitions}
\crefformat{definition}{#2definition~#1#3} 
\Crefformat{definition}{#2Definition~#1#3} 

\crefname{ex}{example}{examples}
\crefformat{example}{#2example~#1#3} 
\Crefformat{example}{#2Example~#1#3} 

\crefname{remark}{remark}{remarks}
\crefformat{remark}{#2remark~#1#3} 
\Crefformat{remark}{#2Remark~#1#3} 

\crefname{convention}{convention}{conventions}
\crefformat{convention}{#2convention~#1#3} 
\Crefformat{convention}{#2Convention~#1#3} 

\crefname{notation}{notation}{notations}
\crefformat{notation}{#2notation~#1#3} 
\Crefformat{notation}{#2Notation~#1#3} 

\crefname{question}{question}{questions}
\crefformat{question}{#2question~#1#3} 
\Crefformat{question}{#2Question~#1#3} 

\crefname{lemma}{lemma}{lemmas}
\crefformat{lemma}{#2lemma~#1#3} 
\Crefformat{lemma}{#2Lemma~#1#3} 

\crefname{proposition}{proposition}{propositions}
\crefformat{proposition}{#2proposition~#1#3} 
\Crefformat{proposition}{#2Proposition~#1#3} 

\crefname{corollary}{corollary}{corollaries}
\crefformat{corollary}{#2corollary~#1#3} 
\Crefformat{corollary}{#2Corollary~#1#3} 

\crefname{theorem}{theorem}{theorems}
\crefformat{theorem}{#2theorem~#1#3} 
\Crefformat{theorem}{#2Theorem~#1#3} 

\crefname{enumi}{}{}
\crefformat{enumi}{(#2#1#3)}
\Crefformat{enumi}{(#2#1#3)}

\crefname{assumption}{assumption}{Assumptions}
\crefformat{assumption}{#2assumption~#1#3} 
\Crefformat{assumption}{#2Assumption~#1#3} 

\crefname{equation}{}{}
\crefformat{equation}{(#2#1#3)} 
\Crefformat{equation}{(#2#1#3)}

\numberwithin{equation}{section}
\renewcommand{\theequation}{\thesection-\arabic{equation}}

\theoremstyle{nonumberplain}
\theoremsymbol{\ensuremath{\blacksquare}}

\newtheorem{proof}{Proof}
\newcommand\pf[1]{\newtheorem{#1}{Proof of \Cref{#1}}}

\newcommand\bR{{\mathbb R}}

\newcommand\bZ{{\mathbb Z}}

\newcommand\cF{{\mathcal F}}

\newcommand\numberthis{\addtocounter{equation}{1}\tag{\theequation}}

\newcommand{\qedhere}{\mbox{}\hfill\ensuremath{\blacksquare}}

\title{(In)equality distance patterns and embeddability into Hilbert spaces}
\author{Alexandru Chirvasitu}

\begin{document}

\date{}

\newcommand{\Addresses}{{
  \bigskip
  \footnotesize

  \textsc{Department of Mathematics, University at Buffalo, Buffalo,
    NY 14260-2900, USA}\par\nopagebreak \textit{E-mail address}:
  \texttt{achirvas@buffalo.edu}

}}

\maketitle

\begin{abstract}
  We show that compact Riemannian manifolds, regarded as metric spaces with their global geodesic distance, cannot contain a number of rigid structures such as (a) arbitrarily large regular simplices or (b) arbitrarily long sequences of points equidistant from pairs of points preceding them in the sequence. All of this provides evidence that Riemannian metric spaces admit what we term loose embeddings into finite-dimensional Euclidean spaces: continuous maps that preserve both equality as well as inequality.

  We also prove a local-to-global principle for Riemannian-metric-space loose embeddability: if every finite subspace thereof is loosely embeddable into a common $\mathbb{R}^N$, then the metric space as a whole is loosely embeddable into $\mathbb{R}^N$ in a weakened sense.
\end{abstract}


\vspace{.5cm}



\section*{Introduction}

The present note is a follow-up on \cite{chi-qmtr-count}, where the following notion was introduced (\cite[Definition 2.2]{chi-qmtr-count}):

\begin{definition}\label{def:le}
  Let $(X,d_X)$ and $(Y,d_Y)$ be two metric spaces. A continuous map $f:X\to Y$ is a {\it loosely isometric} (or just {\it loose}) {\it embedding} if
  \begin{equation*}
    d_X(x,x')\mapsto d_Y(f x,f x')
  \end{equation*}
  is a well-defined one-to-one map on the codomain of $d_X$.

  $(X,d_X)$ is {\it loosely embeddable} (or LE) in $(Y,d_Y)$ if it admits such a $f:X\to Y$, and it is just plain LE (without specifying $Y$) if it is loosely embeddable into some finite-dimensional Hilbert space.
\end{definition}

In other words, $f$ turns (un)equal distances into (un)equal distances respectively. Note in particular that loose embeddings are automatically one-to-one, so they are embeddings. It is the condition of being {\it isometric} that is being loosened.

The concept was originally motivated by the fact that, by (a slight paraphrase of) \cite[Corollary 4.9]{metric}, LE compact metric spaces have {\it quantum automorphism groups}, i.e. they admit universal isometric actions by compact quantum groups. Despite its origin in the non-commutative-geometry considerations central to \cite{chi-qmtr-count}, the notion seems to hold some independent interest of its own. Roughly speaking:

{\it
  Loose embeddability captures the combinatorial patterns of distance (in)equality achievable in Hilbert spaces. 
}

We focus in particular on {\it Riemannian} compact metric spaces, i.e. those obtained by equipping Riemannian manifolds with their global geodesic distance (\Cref{def:riem}).   According to \cite[Theorem 0.2]{cg} compact connected Riemannian manifolds (with or without boundary) always have quantum isometry groups, which in fact coincide with their {\it classical} isometry groups. This means that in the present context we are departing from the {\it initial} motivation for considering loose embeddability, namely the existence of quantum isometry groups. The concept nevertheless suggests some apparently-non-trivial questions in metric and Riemannian geometry.

These seem particularly well suited for loose embeddability, as attested by several results ruling out non-LE metric configurations in the Riemannian context: 

\begin{itemize}
\item \Cref{le:1dim} is the simple remark that 1-dimensional Riemannian manifolds are always loosely embeddable.
\item In \Cref{pr:riemsimpl} we observe that a (compact) Riemannian metric space cannot contain $n$-tuples of equidistant points for arbitrarily large $n$. 
\item Generalizing this, \Cref{pr.riem-flg} shows that compact Riemannian metric spaces do not contain arbitrarily large sets of pairs $\{x_i,y_i\}$ of points with $x_i$ and $y_j$ both equidistant from $x_i$ and $y_i$ for all $j>i$. This rules out (in the Riemannian case) a subtler class of counterexamples to loose metric embeddability given by \Cref{le.no-flg}. 
\end{itemize}
In short, Riemannian compact metric spaces make for poor counterexamples to loose metric embeddability.

\Cref{se:qu} further contains a number of questions related to loose embeddability, and answer a weakened form of \Cref{qu:1} affirmatively in the Riemannian setting: \Cref{pr:riemq1} says, roughly speaking, that if the finite subspaces of a compact Riemannian metric space $(M,d)$ are {\it uniformly} loosely embeddable (i.e. loosely embeddable into Hilbert spaces of uniformly bounded dimension) then $(M,d)$ itself is loosely embeddable in a weak sense.

\subsection*{Acknowledgements}

This work was partially supported by NSF grants DMS-1801011 and DMS-2001128.


\section{Distance configurations in Riemannian metric spaces}\label{se:loose}

It is a natural problem to determine to what extent various classes of metric spaces are loosely embeddable in the sense of \Cref{def:le}. Of special interest, for instance, are Riemannian manifolds equipped with the geodesic metric. \cite{doC92,berg-pan} are good sources for the Riemannian geometry we will peruse.

\begin{definition}\label{def:riem}
  A {\it Riemannian metric space} is a connected Riemannian manifold equipped with the global geodesic metric.
\end{definition}

Unless specified otherwise, all of our metric spaces are assumed compact; we thus often drop that adjective for brevity. Note first:

\begin{lemma}\label{le:1dim}
  One-dimensional compact Riemannian metric spaces are loosely embeddable. 
\end{lemma}
\begin{proof}
  Indeed, 1-dimensional connected compact Riemannian manifolds are isometric to the unit circle, which can be loosely embedded into the plane via its standard origin-centered-unit-circle realization.
\end{proof}

It is observed in \cite[Example 2.3]{chi-qmtr-count} that metric spaces containing {\it regular $n$-simplices} (i.e. point $n$-tuples with equal pairwise distances) for each $n$ cannot be loosely embeddable. The following observation rules this out for Riemannian metric spaces.

\begin{proposition}\label{pr:riemsimpl}
  Let $(M,d)$ be a compact Riemannian manifold with its geodesic metric. There is an upper bound on the number of vertices of a regular simplex in $M$.
\end{proposition}
\begin{proof}
  Now let $v_1$ and $v_2$ be two other vertices of $\Delta$, chosen so that the angle $\varepsilon=\measuredangle v_1v_0v_2$ is sufficiently small (possible for large $n$). Since $M$ is compact there is a global lower bound $K$ for its sectional curvature. By the Toponogov comparison theorem (\cite[\S 6.4.1, Theorem 73]{berg-pan}) the length of $v_1v_2$ is bounded above by the length of the third edge in an isosceles triangle with angle $\varepsilon$ subtending the two edges of equal length $\ell=v_0v_1=v_0v_2$ in the space form \cite[\S 6.3.2]{berg-pan} of constant curvature $K$. This length goes to $0$ as $\varepsilon$ does, contradicting $v_1v_2=\ell$.
\end{proof}

Large regular simplices are not the only obstruction to loose embeddability. The somewhat more sophisticated configurations that pose problems involve, roughly speaking, large sets of points each equidistant to large sets of pairs of points. To make sense of this we need some terminology.

\begin{definition}\label{def.flg}
  Let $n$ be a positive integer.  An {\it $n$-flag of median hyperplanes} is a collection of points
  \begin{equation}\label{eq:nflg}
    \{p_i, q_i,\ 0\le i\le n-1\}
  \end{equation}
  such that
  \begin{equation*}
    d(z,p_s) = d(z,q_s)
  \end{equation*}
  for all $z=p_i$ or $q_i$ with $i>s$. 
\end{definition}

The term `median hyperplane' is meant to invoke the locus of points in a Euclidean space that are equidistant from two given points, while `flag' means chain ordered by inclusion, as in
\begin{equation}\label{eq:1}
  \{p_i,q_i\}_{i\ge 0}\supset \{p_i,q_i\}_{i\ge 1}\supset\cdots.
\end{equation}

The relevance of the concept stems from the following simple remark.

\begin{lemma}\label{le.no-flg}
  A compact metric space containing $n$-flags of median hyperplanes is not LE. 
\end{lemma}
\begin{proof}
  If such a space $(X,d)$ were loosely embeddable in $\bR^d$ say, then each of the sets \Cref{eq:1} would be contained in a hyperplane of $\bR^d$, namely the median hyperplane of (the images in $\bR^d$ of) $p_i$ and $q_i$. These hyperplanes would be orthogonal in the sense that their range projections commute, so any $>d$ of them would intersect trivially.
\end{proof}

On the other hand, Riemannian manifolds can still not be discounted as LE on the basis of \Cref{le.no-flg}.

\begin{theorem}\label{pr.riem-flg}
  A compact Riemannian manifold $(X,d)$ equipped with the geodesic metric cannot contain $n$-flags of median hyperplanes for arbitrarily large $n$. 
\end{theorem}

This will require some amount of preparation.

First, we will have some make some size estimates (for angles, distances, etc.). This raises the usual issue of starting with quantities that are within $\varepsilon>0$ of each other and then obtaining new estimates in terms of $\varepsilon$ such as, say $C\varepsilon$ for some constant $C$. In order to avoid such irrelevancies we make the following

\begin{convention}\label{cv:eps}
  $\varepsilon$ will typically denote a small positive real, and whenever a new small quantity depending on $\varepsilon$ is introduced, we denote it by decorating $\varepsilon$ with the usual symbols used to indicate differentiation. So for instance $\varepsilon'$, $\varepsilon''$, $\varepsilon^{(5)}$, etc. all denote small positive reals depending on $\varepsilon$ in some unspecified fashion.

  The same notational convention applies to other symbols meant to denote small positive reals. 
\end{convention}

In the discussion below we will modify the Riemannian tensor $g$ on a geodesic ball of a Riemannian manifold $(M,g)$ so as to ``flatten'' said ball. The relevant concept is

\begin{definition}\label{def:euc}
  Let $B\subset M$ be a geodesic ball in a Riemannian manifold $M$ with tensor $g$, and suppose we have fixed a coordinate system for $B$. We say that $g$ is {\it $\varepsilon$-Euclidean to order $k$} along $B$ if the derivatives of orders $\le k$ of $g$ within $\varepsilon$ of their usual Euclidean counterparts, uniformly on $B$, in the respective coordinate system.

  We typically omit $k$ from the discussion, simply assuming it is large enough ($k\ge 2$ will do for most of our purposes); for that reason, we abbreviate the phrase as {\it $\varepsilon$-Euclidean}.

  The specific $\varepsilon>0$ will also depend on the chosen coordinates, but we ignore this issue too, as the discussion below will only require $\varepsilon$ sufficiently small, and the various coordinate choices will not affect this.
\end{definition}

As a consequence of the smooth dependence of ODE solutions on the initial data (e.g. \cite[Theorem B.3]{dk}), ``sufficiently Euclidean'' Riemannian metrics in the sense of \Cref{def:euc} have ``sufficiently straight'' geodesics. More formally (keeping in mind \Cref{cv:eps}):

\begin{proposition}\label{pr:str}
Let $(M,g)$ be a Riemannian manifold, $\varepsilon$-Euclidean with respect to some coordinate system. Then, for every geodesic $\gamma$ in $M$, parallel transport of vectors along $\gamma$ does not alter angles by more than $\varepsilon'$
\qedhere
\end{proposition}

\begin{notation}
  Let $M$ be a Riemannian manifold with metric tensor $g$ and geodesic distance $d$. We write
  \begin{equation*}
    \mathrm{inj}(M)\ :=\text{ {\it injectivity radius} of }M
  \end{equation*}
  (\cite[p.271]{doC92} or \cite[p.142, Definition 23]{berg-pan}): the largest number such that all pairs of points less than $\mathrm{inj}(M)$ apart are joined by a unique geodesic segment.
  
For points $p,q$ in a Riemannian manifold $M$ with
\begin{equation*}
  \ell:=d(p,q) < \mathrm{inj}(M)
\end{equation*}
we write
  \begin{equation*}
    \gamma_p^q:[0,\ell]\to M
  \end{equation*}
  for the geodesic arc from $p$ to $q$, parametrized by arclength. We will also abuse notation and denote the image of $\gamma_p^q$ by the same symbol.
\end{notation}

\begin{definition}\label{def:angle}
  Let $p,q$ be points in a Riemannian manifold $M$, less than $\mathrm{inj}(M)$ apart. The {\it angle}
  \begin{equation*}
    \angle(v_p,v_q) 
  \end{equation*}
  between two tangent vectors $v_p\in T_pM$ and $v_q\in T_qM$ is defined by 
  \begin{itemize}
  \item parallel-transporting (\cite[Chapter 2, Proposition 2.6 and Definition 2.5]{doC92} or \cite[p.264, Proposition 61]{berg-pan}) the unit velocity vector $v_q$ to a vector $v\in T_pM$ along $\gamma_p^{q}$;
  \item set
    \begin{equation*}
      \angle(v_p,v_q) := \text{ angle between }v_p\text{ and }v,
    \end{equation*}
    computed in $T_p(M)$ as usual, via the Riemannian tensor. 
  \end{itemize}
  
  For points $p$, $q$, $p'$, $q'$ in $M$, each two less than $\mathrm{inj}(M)$ apart, the angle $\angle(\gamma_p^q,\gamma_{p'}^{q'})$ is the angle (defined as above) between the unit velocity vectors $(\gamma_p^q)'(0)$ and $(\gamma_{p'}^{q'})'(0)$.
\end{definition}

\begin{remark}
  Although \Cref{def:angle} appears to bias one of the pairs $p,q$ and $p',q'$ over the other, the notion is in fact symmetric: because parallel transport is an isometry between tangent spaces, whether we parallel-transport
  \begin{equation*}
    (\gamma_{p'}^{q'})'(0)\text{ to }T_pM
  \end{equation*}
  or 
  \begin{equation*}
    (\gamma_{p}^{q})'(0)\text{ to }T_{p'}M
  \end{equation*}
  does not affect the value of the angle. 
\end{remark}

For an $n$-dimensional Riemannian metric space $(M,d=d_M)$ with a basepoint $z\in M$ we will consider small geodesic balls
\begin{equation*}
  B_r=B_r(z):=\{q\in M\ |\ d(z,q)\le r\}
\end{equation*}
centered at $z$, parametrized with {\it normal coordinates} \cite[\S 4.4.1]{berg-pan} $x^i$, $1\le i\le n$ (so $z$ is identified with the origin $(0,\cdots,0)$). Recall that this means the geodesics emanating from $z$ are identified with straight line segments.

Having fixed such a coordinate system, we can speak about segments in $B$, angles between those segments, etc.; it will be clear from context when these are actual segments in the ambient $\bR^n$ housing $B$ rather than, say, {\it geodesic} segments in $M$.   

Typically, the radius $r$ decorating $B_r$ will be small. We will occasionally have to normalize the Riemannian metric in $B_r$, scaling distances from the origin $z=0\in B$ by $\frac 1r$ so that the new ball ${}_nB_r$ (`n' for `normalized') has radius $1$.

This normalization procedure has the effect of ``flattening'' the Riemannian metric, in the sense that the Riemannian structure can be made arbitrarily $\varepsilon$-Euclidean (\Cref{def:euc}) as $r\to 0$.  


In the discussion below, for a Riemannian manifold $M$ with geodesic metric $d=d_M$, we write
\begin{equation}\label{eq:eta}
  \eta(p,q) = \eta_M(p,q):=d(x,y)^2
\end{equation}
for the squared-distance function (the notation matches that in \cite{nic-rand} for instance, where this function features prominently).

\begin{lemma}\label{le:grad}
  Let $M$ be a Riemannian manifold and $B=B_r(z)$ a sufficiently small geodesic ball equipped with normal coordinates around $z\in M$. Let also $p\in B$ be a point and consider the function
  \begin{equation*}
    \psi:x\mapsto \eta(x,p).
  \end{equation*}
  with $\eta$ as in \Cref{eq:eta}. Denoting by $v\in T_zM$ the unit vector tangent to the geodesic $z\to p$, the gradient $\nabla \psi$ at $z$ equals $-2d(z,p)v$. 
\end{lemma}
\begin{proof}
  This is immediate after choosing a normal coordinate system around $p$, whereupon $\psi$ becomes
  \begin{equation*}
    \psi:(x^1,\cdots,x^n)\to \sum_{i=1}^n (x^i)^2. 
  \end{equation*}
\end{proof}

\pf{pr.riem-flg}
\begin{pr.riem-flg}
  Suppose we {\it do} have arbitrarily large flags of median hyperplanes in our compact Riemannian space $(M,d)$. Since $M$ is compact, we can assume that some large flag \Cref{eq:nflg} is contained entirely within some small geodesic ball $B_r$ centered at a point $z:=p_n$ constituting the flag.

  We can assume $r$ is small enough that the normalized ball ${}_nB_{r}$ is $\varepsilon$-Euclidean in the sense of \Cref{def:euc}. Furthermore, because the size of the flag can also be chosen arbitrarily large, we can also assume that 
  \begin{equation*}
    \angle\left(\gamma_{p_i}^{q_i},\ \gamma_{p_j}^{q_j}\right)<\varepsilon',\ \forall 0\le i\ne j<n
  \end{equation*}

  Henceforth, it will be enough to work with $p_i$ and $q_i$ for $i=0,1$. By the flag condition, both $p_1$ and $q_1$ are equidistant from $p_0$ and $q_0$. Additionally, we have
  \begin{equation}\label{eq:pq01}
    \angle\left(\gamma_{p_1}^{q_1},\ \gamma_{p_0}^{q_0}\right)<\varepsilon'
  \end{equation}

  Furthermore, because the metric is $\varepsilon$-Euclidean, the unit-length velocity vectors $v_x$ along
  \begin{equation*}
    \gamma:=\gamma_{p_1}^{q_1}
  \end{equation*}
  stay within an angle of $\varepsilon''$ of the initial velocity vector $(\gamma_{p_1}^{q_1})'(0)$ (by \Cref{pr:str}), so \Cref{eq:pq01} implies that
  \begin{equation}\label{eq:vx}
    \angle\left(v_x,\ (\gamma_{p_0}^{q_0})'(0)\right)<\varepsilon^{(3)},\ \forall x\in \gamma_{p_1}^{q_1}. 
  \end{equation}

For each $x\in \gamma$, we saw in \Cref{le:grad} that the gradient of the function
\begin{equation}\label{eq:distdiff}
  \psi:x\mapsto d(x,p_0)^2 - d(x,q_0)^2 
\end{equation}
  is
  \begin{equation}\label{eq:realvect}
    2d(x,q_0)(\gamma_x^q)'(0) - 2d(x,p_0)(\gamma_x^p)'(0).
  \end{equation}

This is the parallel transport of $2\overrightarrow{p_0q_0}$ to $x$ in the usual, Euclidean metric, so by our assumption that the original metric is $\varepsilon$-Euclidean the angle between \Cref{eq:realvect} and $(\gamma_{p_0}^{q_0})'(0)$ is $<\varepsilon^{(4)}$. To summarize, we have
  \begin{itemize}
  \item a small angle between each gradient $\nabla_x \psi$ of $\psi$ along $\gamma$, given by \Cref{eq:realvect}, and $(\gamma_{p_0}^{q_0})'(0)$;  
  \item a small angle between the latter and the unit tangent vectors $v_x$ at $x\in \gamma$, by \Cref{eq:vx}.
  \end{itemize}
  In particular, at each $x$ along $\gamma$ the gradient $\nabla_x\psi$ and the velocity along $\gamma$ have positive inner product. This means that the function $\psi$ in \Cref{eq:distdiff} increases strictly along the geodesic $\gamma$, contradicting the fact that it must take the value $0$ at both endpoints $p_1$ and $q_1$.
\end{pr.riem-flg}

\section{Questions}\label{se:qu}

\Cref{pr:riemsimpl} and \Cref{pr.riem-flg} seem to suggest that compact Riemannian metric spaces are particularly amenable to loose metric embeddability. I do not know whether they are always LE, but that problem decomposes naturally: first,

\begin{question}\label{qu:1}
Let $(X,d)$ be a compact metric space and $N\in \bZ_{>0}$ a positive integer such that every finite subspace of $(X,d)$ is loosely embeddable into $\bR^N$. Does it follow that $(X,d)$ itself is LE? 
\end{question}

In other words, does {\it uniform} loose embeddability for the finite subspaces of $(X,d)$ entail the LE property for $X$ as a whole?

Secondly, to circle back to the Riemannian context:

\begin{question}\label{qu:2}
  Do compact Riemannian metric spaces satisfy the hypothesis of \Cref{qu:1}?
\end{question}

We conclude with a partial answer to \Cref{qu:1}. First, we need

\begin{definition}\label{def:weakle}
A metric space $(X,d_X)$ is {\it weakly loosely embeddable} (or {\it weakly LE}) in the metric space $(Y,d_Y)$ if there is an injective map $f:X\to Y$ satisfying only the backwards implication in the biconditional implicit in \Cref{def:le}:   
  \begin{equation}\label{eq:wle}
    d_Y(fx,fx') = d_Y(fz,fz') \Leftarrow d_X(x,x') = d_X(z,z').
  \end{equation}
\end{definition}

\begin{theorem}\label{pr:riemq1}
Under the hypotheses of \Cref{qu:1}, a compact Riemannian metric space is weakly LE in $\bR^N$. 
\end{theorem}
\begin{proof}
  Let $(M,d)$ be a compact Riemannian manifold with its geodesic metric and denote by $(\cF,\subseteq)$ the poset of finite subsets $F\subset M$ (ordered by inclusion). For each $F\in \cF$ we fix a map
  \begin{equation*}
    \psi_F:F\to B:=\text{ origin-centered unit ball in }\bR^N
  \end{equation*}
  such that
  \begin{itemize}
  \item $\psi_F$ is a loose embedding of $(F,d)$, rescaled if needed so as to ensure it lands in the ball $B$;
  \item the diameter of $\psi_F(F)$ is precisely $1$, with $\psi_F p=0$ and $\psi_F q$ on the unit sphere $\partial B$ for some $p,q\in F$.
  \end{itemize}
  This gives us an $\cF$-indexed net \cite[Chapter 3, p.187]{mun} $\psi_F$ of maps $F\to B$, and since
  \begin{itemize}
  \item $B$ is compact;    
  \item every element $p\in M$ belongs to sufficiently large $F\in \cF$, i.e. to the upward-directed set
    \begin{equation*}
      \{F\in \cF\ |\ p\in F\},
    \end{equation*}
  \end{itemize}
  we can take the pointwise limit
  \begin{equation*}
    \psi(p):= \lim_{\cF}\psi_{F}(p)\in B
  \end{equation*}
  to obtain a map $\psi:M\to B$. it remains to prove that $\psi$
  \begin{enumerate}[(a)]
  \item\label{item:2} satisfies the weak LE condition \Cref{eq:wle};
  \item\label{item:1} is continuous;  
  \item\label{item:3} is one-to-one.     
  \end{enumerate}

  {\bf \Cref{item:2}: condition \Cref{eq:wle}.}
  We want to prove that
  \begin{equation}\label{eq:back}
    |\psi x-\psi x'| = |\psi z-\psi z'| \Leftarrow d_M(x,x') = d_M(z,z')
  \end{equation}
  holds; this follows by passing to the limit over $F\in \cF$ in the analogous implication for the partially-defined maps $\psi_F:F\to B$.

  We can now define a map
  \begin{equation}\label{eq:phi}
    \varphi:(\text{set of distances }d_M(p,q))\to \bR_{\ge 0}
  \end{equation}
  by
  \begin{equation}\label{eq:phi}
    \varphi(d_M(p,q)) = |\psi p-\psi q|. 
  \end{equation}
  We define the maps $\varphi_F$, $F\in \cF$ similarly, substituting $\psi_F$ for $\psi$ in \Cref{eq:phi}.

  {\bf \Cref{item:1}: $\psi$ is continuous.} We have to argue that
  \begin{equation*}
    \lim_{d\to 0}\varphi(d) = 0.
  \end{equation*}
  If not, we can find a subnet $(F_{\alpha})_{\alpha}$ of $\cF$ and points $p_{\alpha}, q_{\alpha}\in F_{\alpha}$ such that
  \begin{equation*}
    d_M(p_{\alpha},q_{\alpha})\to 0
  \end{equation*}
  but
  \begin{equation}\label{eq:eps}
    \varepsilon:=\inf_{\alpha} |\psi_{F_{\alpha}} p_{\alpha}-\psi_{F_{\alpha}} q_{\alpha}| > 0;
  \end{equation}
  we abbreviate
  \begin{equation*}
    \psi_{\alpha}:=\psi_{F_{\alpha}},
  \end{equation*}
  and similarly for $\varphi$. 

  If $\ell>0$ is sufficiently small (smaller than the injectivity radius of $M$, for instance \cite[Definition following Theorem III.2.3]{chav}), then $(M,d_M)$ contains geodesic triangles with edges
  \begin{equation*}
    \ell,\ \ell,\ t
  \end{equation*}
  for every $2\ell>t>0$. This can easily be seen, for instance, by continuously decreasing the angle between two length-$\ell$ geodesic rays based at a point from $\pi$ to $0$; the distance between the extremities of those geodesic rays will then decrease continuously from $2\ell$ to $0$.

  Now fix some $\ell>0$, sufficiently small. We will have $d_M(p_{\alpha},q_{\alpha})<2\ell$ for sufficiently large $\alpha$, and hence, by the preceding remark, we can find geodesic triangles in $M$ with edges $\ell$, $\ell$ and $d_M(p_\alpha,q_{\alpha})$ (assuming also that $\alpha$ is large enough to ensure that $F_{\alpha}$ contains the tip of that isosceles geodesic triangle).

  Applying $\psi_{\alpha}$, we have a triangle in $B$ with edges
  \begin{equation*}
   \varphi_{\alpha}(\ell),\ \varphi_{\alpha}(\ell),\ \varphi_{\alpha}(d_M(p_{\alpha},q_{\alpha})). 
  \end{equation*}
  In particular, we have
  \begin{equation}\label{eq:eps2}
    \varphi_{\alpha}(\ell) \ge \frac{\varphi_{\alpha}(d_M(p_{\alpha},q_{\alpha}))}2 \ge \frac{\varepsilon}2 > 0
  \end{equation}
  by \Cref{eq:eps}. Since $\ell>0$ was arbitrary (so long as it was small enough), this means that by passing to large enough $\alpha$ we can find arbitrarily large finite subsets $F$ of $M$, of girth $\ge \ell$ (i.e. so that all pairs of points are at least $\ell$ apart), and hence so that (by \Cref{eq:eps2})
  \begin{equation*}
    |\psi p-\psi q|\ge \frac \varepsilon 2,\ \forall p,q\in F. 
  \end{equation*}
  Since the cardinality of $F$ (and hence that of $\psi(F)$) can be made arbitrarily large, we are contradicting the compactness of $B$. This completes the proof of \Cref{item:1} above. 
  
  {\bf \Cref{item:3}: $\psi$ is injective.} Suppose not. In a sense, this means we are in precisely the opposite situation to that encountered in the proof of part \Cref{item:1}: there is a subnet $(F_{\alpha})_{\alpha}$ of $\cF$ with  points $p_{\alpha},\ q_{\alpha}\in F_{\alpha}$ such that
  \begin{align*}
    \ell:=\inf_{\alpha} d_M(p_{\alpha},q_{\alpha}) &> 0 \\
    \inf_{\alpha} |\psi_{\alpha} p_{\alpha} - \psi_{\alpha}q_{\alpha}| &= 0. \numberthis \label{eq:conv0}
  \end{align*}
  By compactness, we can also assume $p_{\alpha}$ and $q_{\alpha}$ are convergent and hence in particular that the distances
  \begin{equation*}
    \ell_{\alpha}:=d_M(p_{\alpha},q_{\alpha})
  \end{equation*}
  are as well. For sufficiently small $t>0$ there are triangles in $M$ with edges
  \begin{equation*}
    \ell_{\alpha},\ \ell_{\alpha},\ t
  \end{equation*}
  for all $\alpha$ (consider two geodesic rays of length $\ell_{\alpha}$ with common origin, subtending small angles at said origin). But then an application of one of the $\psi_{\alpha}$ will yield a triangle with edges
  \begin{equation*}
    \varphi_{\alpha}(\ell_{\alpha}),\ \varphi_{\alpha}(\ell_{\alpha}),\ \varphi_{\alpha}(t)
  \end{equation*}
  with $\varphi$ and $\varphi_{\alpha}:=\varphi_{F_{\alpha}}$ as in \Cref{eq:phi} and subsequent discussion, meaning that
  \begin{equation*}
    \varphi_{\alpha}(t)\le \frac{\varphi_{\alpha}(\ell_{\alpha})}{2}.  
  \end{equation*}
  Since the right hand side converges to zero by \Cref{eq:conv0}, we conclude that $\varphi(t)=0$ for {\it all} sufficiently small $t>0$. In other words,
  \begin{equation}\label{eq:psicollapse}
    \psi \text{ identifies any two points that are sufficiently close.}
  \end{equation}  
  Now, for each $\alpha$ we also have, by assumption, points $x_{\alpha}, y_{\alpha}$ in $F_{\alpha}$ that achieve distance $1$ upon applying $\psi_{\alpha}$:
  \begin{equation*}
    |\psi_{\alpha} x_{\alpha} - \psi_{\alpha}y_{\alpha}| = 1
  \end{equation*}
  By compactness, passage to a subnet if necessary allows us to assume that $x_{\alpha}$ and $y_{\alpha}$ converge to $x$ and $y$ in $M$ respectively, and limiting over $\alpha$ produces $|\psi x - \psi y| = 1$.

  For some distance $t>0$ small enough to qualify for \Cref{eq:psicollapse} we can find a broken geodesic consisting of some finite number $N$ of length-$t$ segments
  \begin{equation*}
    x=:p_0\to p_1,\quad p_1\to p_2,\ \cdots,\ p_{N-1}\to p_N:=y
  \end{equation*}
  connecting $x$ and $y$. Applying $\psi$ we similarly obtain a broken geodesic consisting of $N$ length-$\varphi(t)$ segments connecting $x$ and $y$, but the latter are distance $1$ apart while $\varphi(t)=0$ by \Cref{eq:psicollapse}. This gives the contradiction we seek and finishes the proof.
\end{proof}

\begin{remark}\label{re:corn}
  As noted above, compact connected Riemannian manifolds are known to have quantum isometry groups and the latter are classical. We note however that the results proven above for Riemannian manifolds involve only {\it local} considerations. This means that, for instance, \Cref{pr:riemsimpl,pr.riem-flg,pr:riemq1} apply to submanifolds {\it with corners} of compact Riemannian manifolds.

  Recall \cite[Definition 2.1]{joy} that the latter are manifolds modeled as usual, via atlases, on the spaces $\bR_{\ge 0}^k\times \bR^{n-k}$. One can obtain interesting metric spaces by
  \begin{itemize}
  \item starting with a Riemannian manifold;
  \item cutting out a domain bounded by hypersurfaces intersecting transversally ;
  \item restricting the global geodesic metric to that domain. 
  \end{itemize}
  As soon as one goes beyond manifolds with boundary such spaces are not covered by the main results of \cite{cg}.
\end{remark}



\begin{thebibliography}{10}

\bibitem{berg-pan}
Marcel Berger.
\newblock {\em A panoramic view of {R}iemannian geometry}.
\newblock Springer-Verlag, Berlin, 2003.

\bibitem{chav}
Isaac Chavel.
\newblock {\em Riemannian geometry---a modern introduction}, volume 108 of {\em
  Cambridge Tracts in Mathematics}.
\newblock Cambridge University Press, Cambridge, 1993.

\bibitem{chi-qmtr-count}
Alexandru Chirvasitu.
\newblock Quantum isometries and loose embeddings, 2020.
\newblock arXiv:2004.09962.

\bibitem{cg}
Alexandru Chirvasitu and Debashish Goswami.
\newblock Existence and {R}igidity of {Q}uantum {I}sometry {G}roups for
  {C}ompact {M}etric {S}paces.
\newblock {\em Comm. Math. Phys.}, 380(2):723--754, 2020.

\bibitem{doC92}
Manfredo~Perdig{\~a}o do~Carmo.
\newblock {\em Riemannian geometry}.
\newblock Mathematics: Theory \& Applications. Birkh\"auser Boston, Inc.,
  Boston, MA, 1992.
\newblock Translated from the second Portuguese edition by Francis Flaherty.

\bibitem{dk}
J.~J. Duistermaat and J.~A.~C. Kolk.
\newblock {\em Lie groups}.
\newblock Universitext. Springer-Verlag, Berlin, 2000.

\bibitem{metric}
Debashish Goswami.
\newblock Existence and examples of quantum isometry groups for a class of
  compact metric spaces.
\newblock {\em Adv. Math.}, 280:340--359, 2015.

\bibitem{joy}
Dominic Joyce.
\newblock On manifolds with corners.
\newblock In {\em Advances in geometric analysis}, volume~21 of {\em Adv. Lect.
  Math. (ALM)}, pages 225--258. Int. Press, Somerville, MA, 2012.

\bibitem{mun}
James~R. Munkres.
\newblock {\em Topology}.
\newblock Prentice Hall, Inc., Upper Saddle River, NJ, 2000.
\newblock Second edition of [ MR0464128].

\bibitem{nic-rand}
Liviu~I. Nicolaescu.
\newblock Random morse functions and spectral geometry, 2012.
\newblock arXiv:1209.0639.

\end{thebibliography}
\def\polhk#1{\setbox0=\hbox{#1}{\ooalign{\hidewidth
  \lower1.5ex\hbox{`}\hidewidth\crcr\unhbox0}}}

\addcontentsline{toc}{section}{References}

\Addresses

\end{document}